\documentclass{amsart}

 \newtheorem{Theorem}{Theorem}[section]
 
 \newtheorem{lemma}[Theorem]{Lemma}
 
 \theoremstyle{definition}
 \newtheorem{definition}[Theorem]{Definition}
 \theoremstyle{remark}

 \numberwithin{equation}{section}

\renewcommand{\Re}{\operatorname{Re}}

\begin{document}

\title[Zero-two law for cosine families]
 {Zero-two law for cosine families}


\author[Schwenninger]{Felix Schwenninger}
\address{%
Department of Applied Mathematics, \\ University of Twente, P.O. Box 217, \\ 7500 AE Enschede, The Netherlands}

\email{f.l.schwenninger@utwente.nl}

\thanks{The first named author has been supported by the Netherlands Organisation for Scientific Research (NWO), grant no. 613.001.004.}
\author[Zwart]{Hans Zwart}
\address{Department of Applied Mathematics, \\ University of Twente, P.O. Box 217, \\ 7500 AE Enschede, The Netherlands}
\email{h.j.zwart@utwente.nl}


\date{September 11, 2014}

\begin{abstract} 
For $\left(C(t)\right)_{t \geq 0}$ being a strongly continuous cosine family on a Banach space, we show that the estimate $\limsup_{t\to 0^{+}}\|C(t) - I\| <2$ implies that $C(t)$ converges to  $I$ in the operator norm. This implication has become known as the zero-two law. We further prove that the stronger assumption of $\sup_{t\geq0}\|C(t)-I\|<2$ yields that $C(t)=I$ for all $t\geq0$. Additionally, we give alternative proofs for similar results for $C_{0}$-semigroups.
\end{abstract}

\maketitle
\section{Introduction}

Let $\left( T(t) \right)_{t \geq 0}$ denote a strongly continuous semigroup on the Banach space $X$ with infinitesimal generator $A$. 
It is well-known that the inequality
\begin{equation}\label{eq0}
\limsup_{t\to0^{+}}\|T(t) - I\| <1,
\end{equation}
implies that the generator $A$ is a bounded operator, see e.g.\
\cite[Remark 3.1.4]{Staffans}. Or equivalently, that the semigroup is
uniformly continuous (at 0), i.e.,
\begin{equation}
\label{eq:1a}
\limsup_{t\to0^{+}}\|T(t)-I\|=0.
\end{equation}
This has become known as \textit{zero-one} law for
semigroups. Surprisingly, the same law holds for general semigroups on
semi-normed algebras, i.e., (\ref{eq0}) implies (\ref{eq:1a}), see
e.g.\ \cite{Esterle04}. For a nice overview and related
results, we refer the reader to \cite{ChEP14}.

In this paper we study the zero-two law for strongly continuous cosine families on a
Banach space, i.e.\ whether
\begin{equation}
	\limsup_{t\to0^{+}}\|C(t) - I\| <2 \quad\text{ implies that }\quad  \limsup_{t\to0^{+}}\|C(t) - I\| = 0.
\end{equation}
This implication is known if the Banach space is UMD, see
\cite[Corollary 4.2]{Fackler2013}, hence, in particular for Hilbert
spaces. On the other hand the $0-3/2$ law, i.e.
\begin{equation*}
	\limsup_{t\to0^{+}}\|C(t) - I\| <\frac{3}{2} \quad\text{ implies that }\quad  \limsup_{t\to0^{+}}\|C(t) - I\| = 0,
\end{equation*}
holds for cosine families on general Banach spaces as was proved by W.~
Arendt in \cite[Theorem 1.1 in Three Line
Proofs]{UlmerSeminare2012}. The result even holds without assuming
that the cosine family is strongly continuous. In the same work,
Arendt poses the question whether the zero-two law holds for cosine
families, \cite[Question 1.2 in Three Line
Proofs]{UlmerSeminare2012}. The following theorem answers this
question positively for strongly continuous cosine families. For its
proof and the definition of a cosine family we refer to Section
\ref{sec:2}.
\begin{Theorem}
\label{Tm:1.1}
Let $\left(C(t)\right)_{t\geq 0}$ be a strongly continuous cosine
family on the Banach space $X$. Then
\begin{equation}\label{eq:2}
  \limsup_{t\to 0^{+}}\|C(t)- I\|<2,
\end{equation}
implies that $\lim_{t\to 0^{+}}\|C(t)-I\|=0$.
\end{Theorem}
By taking $X = \ell^2$ and 
\[
   C(t) = \left( \begin{array}{cccccc} \cos(t) & 0 & \cdots\\
  0 & \cos(2t) & 0 & \cdots \\
  \vdots & & \ddots &&
  \end{array} \right),
\]
it is easy to see that this result is optimal.
Whether one can get rid of the
assumption that the cosine family is strongly continuous remains open.

The zero-one law for semigroups and the zero-two law for cosine
families tells something about the behaviour near $t=0$. Instead of
studying the behaviour around zero, we could study the behaviour on the
whole time axis. A result dating back to the sixties is the following;
for a semigroup the assumption 
\begin{equation}
  \label{eq:1}
  \sup_{t\geq0}\|T(t) - I\| <1,
\end{equation}
implies that $T(t)=I$ for all $t\geq 0$, see e.g.\ Wallen
\cite{Wallen} and Hirschfeld \cite{Hirschfeld}. This result seems not
to be well-known among researchers working in the area of strongly
continuous semigroup. The corresponding result for cosine families,
i.e.,
\begin{equation}
\label{eq:1.6}
	\sup_{t\in {\mathbb R} } \|C(t) - I\| < 2 \quad\text{ implies that }\quad  C(t) = I
\end{equation}
is hardly studied at all. We prove this result for strongly continuous
cosine families on Banach spaces. This result is strongly motivated by
the recent work of A.~Bobrowski and W.~Chojnacki. In \cite[Theorem
4]{BobrowskiApprox}, they showed that if $r< \frac{1}{2}$, where
\begin{equation}
\label{eq:1new}
  r = \sup_{t\geq0}\|C(t)-\cos(at)I\|,
\end{equation}
then $C(t) = \cos(at)I$ for all $t \geq 0$. They used this to conclude
that scalar cosine families are isolated points within the space of
bounded strongly continuous cosine families acting on a fixed Banach
space, equipped with the supremum norm. 

Hence we show that for $a=0$ the $r$ can be chosen be 2, provided $C$
is strongly continuous. We remark
that by using the proof idea in \cite[Theorem 1.1 in Three Line
 Proofs]{UlmerSeminare2012} the implication 
\[
	\sup_{t\in {\mathbb R} } \|C(t) - I\| < r \quad\text{ implies that }\quad  C(t) = I
\]
holds for $r <\frac{3}{2}$ for any cosine family. While this paper was
being revised, we heard that A.~Bobrowski, W.\ Chojnacki and
A.~Gregosiewicz showed that for $a\neq 0$ the implication
\begin{equation}
\label{eq:ar}
	\sup_{t\in {\mathbb R} } \|C(t) - \cos(at)I \| < r
        \quad\text{ implies that }\quad  C(t) = \cos(at) I
\end{equation}
holds for general cosine families with $r=\frac{8}{3\sqrt{3}}$. This
constant is optimal, as can be directly seen by choosing $C(t)=
\cos(3at)I$. In \cite{ScZw-Ar14} we wrongly claimed that
$r=2$ was the optimal constant.

The lay-out of this paper is as follows. 
In Section \ref{sec:2} we prove the zero-two law for strongly
continuous cosine families, i.e., Theorem \ref{Tm:1.1} is
proved. In Section \ref{sec:3} we prove the implication
(\ref{eq:1.6}). Furthermore, we give elementary, alternative proofs
for strongly continuous semigroups. Throughout the paper, we use
standard notation, such as $\sigma(A)$ and $\rho(A)$ for the spectrum
and resolvent set of the
operator $A$, respectively. Furthermore, for $\lambda  \in \rho(A)$,
$R(\lambda,A)$ denotes $(\lambda I - A)^{-1}$.

\section{The zero-two law at the origin}
\label{sec:2}

In this section we prove that for the strongly
continuous cosine family $C$ on the Banach space $X$ Theorem \ref{Tm:1.1} holds; i.e.,
\[
	\limsup_{t\to0^{+}}\|C(t) - I\| <2 \quad\text{ implies that }\quad  \limsup_{t\to0^{+}}\|C(t) - I\| = 0.
\]
However, before we do so, we first recall the definition of a strongly
continuous cosine family. For more information we refer to \cite{ABHN}
or \cite{Fattorini69I}.
\begin{definition}
\label{def:cos}
 A family $C= \left( C(t) \right)_{t \in {\mathbb R}}$ of bounded
   linear operators on $X$ is called a {\em cosine family}\/ when the
   following two conditions hold
   \begin{enumerate}
   \item $C(0)=I$, and
   \item For all $t,s \in {\mathbb R}$ there holds
     \begin{equation}
       \label{eq:3}
       2C(t)C(s) = C(t+s) + C(t-s).
     \end{equation}
   \end{enumerate}
   It is defined to be {\em strongly continuous}, when for all $x\in
   X$ and all $t\in {\mathbb R}$ we
   have
   \[
      \lim_{h \rightarrow 0} C(t+h)x=C(t)x.
   \]
\end{definition}

Similar as for strongly continuous semigroups we can define the infinitesimal
generator.
\begin{definition}
\label{def:gen}
Let $C$ be a strongly continuous cosine family, then the
{\em infinitesimal generator}\ $A$ is defined as
\[
   Ax = \lim_{t \rightarrow 0} \frac{2(T(t)x-x)}{t^2}
\]
with its domain consisting of those $x \in X$ for which this limit exists.
\end{definition}

This infinitesimal generator is a closed, densely defined operator. For
the proof of Theorem \ref{Tm:1.1}, the following well-known estimates are
needed.  For a proof we refer to Lemma 5.5 and 5.6 in \cite{Fattorini69I} .
\begin{lemma}
\label{le:CosLapl}
Let $C$ be a strongly continuous cosine family with generator $A$. Then, there exists $\omega\geq0$ and $M\geq1$ such that  
\begin{equation}
\|C(t)\|\leq Me^{\omega t}\qquad \forall t\geq0.
\end{equation}
Furthermore, for $\Re\lambda>\omega$ we have $\lambda^{2}\in\rho(A)$ and
\begin{equation}
   \|\lambda^{2}R(\lambda^{2},A)\| \leq M\cdot \frac{|\lambda|}{\Re\lambda-\omega}.
\end{equation}
\end{lemma}

Hence the above lemma shows that the spectrum of $A$ must lie within
the parabola $\{ s \in {\mathbb C} \mid s= \lambda^2 \mbox{ with } \Re\lambda =
\omega\}$. To study the spectral properties of the points within this
parabola, we can use the following lemma.
\begin{lemma}\label{le:CosId}
Let $C$ be a strongly continuous cosine family on the Banach space $X$
and let $A$ be its generator. Then, for $\lambda\in\mathbb{C}$ and
$s\in\mathbb{R}$ there holds 
\begin{enumerate}
\item \label{le:CosId1} $S(\lambda,s)$ defined by
          \begin{equation}
            \label{eq:4}
            S(\lambda,s)x=\int_{0}^{s}\sinh(\lambda(s-t))C(t)x\ dt, \quad x\in X,
          \end{equation}
	is a linear and bounded operator on $X$ and its norm satisfies
	\begin{equation}\label{eq:bddS}
	\|S(\lambda,s)\| \leq \sup_{t\in[0,|s|]}\|C(t)\|\cdot
        \frac{\sinh(|s|\Re\lambda)}{\Re\lambda}.
	\end{equation}
\item \label{le:CosId2} For $x \in X$ we have $S(\lambda,s)x\in D(A)$,
	\begin{align}\label{eq:CosId}
	(\lambda^{2}I - A)S(\lambda,s)x=\lambda (\cosh(\lambda s)I - C(s))x.
	\end{align}
	 Furthermore, $S(\lambda,s)A\subset AS(\lambda,s)$.
\item The bounded operators \label{le:CosId3}$S(\lambda,s)$ and $C(s)x-\cosh(\lambda s)I$ commute.
\item \label{le:CosId4}If $\lambda\neq0$ and $\cosh(\lambda s)\in\rho(C(s))$, then $\lambda^{2}\in\rho(A)$ and
		\begin{align}
		\|R(\lambda^{2},A)\| \leq {}& \frac{1}{|\lambda|}\cdot \|S(\lambda,s)\|\cdot \|R(\cosh(\lambda s),C(s))\|\notag \\
		 \leq {}& \sup_{t\in[0,|s|]}\|C(t)\|\cdot  \frac{2|s| e^{|s\Re\lambda|}}{|\lambda|} \cdot \|R(\cosh(\lambda s),C(s))\|.\label{eq:CosId2}
		\end{align}
\end{enumerate}
\end{lemma}
\begin{proof}
  We begin by showing item \ref{le:CosId1}. Since the cosine family is
  strongly continuous, the integral in (\ref{eq:4}) is
  well-defined. Hence $S(\lambda,s)$ is well defined and linear. So
  it remains to consider
	\begin{align*}
		\|S(\lambda,s)x\| \leq{}& \sup_{t\in[0,|s|]}\|C(t)\|\cdot\|x\|\cdot \int_{0}^{|s|}|\sinh(\lambda t)|\ dt\\
						={}&\sup_{t\in[0,|s|]}\|C(t)\|\cdot\|x\|\cdot \frac{1}{2}\int_{0}^{|s|}|e^{\lambda t}-e^{-\lambda t}| \ dt\\	
						\leq{}&\sup_{t\in[0,|s|]}\|C(t)\|\cdot\|x\|\cdot \frac{e^{|s|\Re\lambda}-e^{-|s|\Re\lambda}}{2\Re\lambda}.
		\end{align*}
Since by definition, the last fraction equals
$\frac{\sinh(|s|\Re\lambda)}{\Re\lambda}$, the inequality
(\ref{eq:bddS}) is shown.

\noindent
{\em Item \ref{le:CosId2}}.\ See  \cite[Lemma4]{Nagy74Sz}. 

\noindent
{\em Item \ref{le:CosId3}}.\ This is clear, since $C(t)$ and $C(s)$ commute for $s,t
\in {\mathbb R}$.

\noindent
{\em Item \ref{le:CosId4}}.\ We define the bounded operator
\[
  B=\frac{1}{\lambda}S(\lambda,s)R(\cosh(\lambda s),C(s)).
\]
By item \ref{le:CosId2}., we see that
$(\lambda^{2} I - A )B=I$.
By item \ref{le:CosId3}., we get that
$B=\frac{1}{\lambda}R(\cosh(\lambda s),C(s))S(\lambda,s)$. Thus, again
by \ref{le:CosId2}., $B(\lambda^{2} I -A )x=x$ for $x\in D(A)$. Hence,
$\lambda^{2}\in\rho(A)$ and first inequality of (\ref{eq:CosId2})
follows. By using the power series of the exponential function, it is
easy to see that $\frac{\sinh(|s|\Re\lambda)}{\Re\lambda}\leq
2|s|e^{|s\Re\lambda|}$. Combining this with (\ref{eq:bddS}) gives the
second inequality in (\ref{eq:CosId2}).
\end{proof}

With the use of the above lemma we can show that the spectrum of $A$ is
contained in the intersection of a ball and a parabola, provided (\ref{eq:2}) holds, i.e.,
$\limsup_{t\to0^{+}}\|C(t)-I\|<2$.
\begin{lemma}\label{le:Pinv}
Let $C$ be a strongly continuous cosine family on the Banach space $X$ with generator $A$. Assume that there exists $c>0$ such that 
\begin{equation}\label{le:Pinveqass}
\limsup_{t\to0^{+}}\|C(t)-I\|<c<2.
\end{equation}
Then, there exists $M_{c},r_{c}>0$ and $\phi_{c}\in(0,\frac{\pi}{2})$ such that 
\begin{equation}
\label{eq:Rc}
 \mathcal{R}_{c}:=\left\{\lambda^{2} \mid \lambda\in\mathbb{C}, |\lambda|>r_{c},|\arg(\lambda)|\in\left(\phi_{c},\frac{\pi}{2}\right]\right\} \subset\rho(A),
\end{equation}
and 
\begin{equation}
\forall \mu\in \mathcal{R}_{c}\qquad \|\mu R(\mu,A)\|\leq M_{c}.
\end{equation}
\end{lemma}
\begin{proof}
First, we note that by (\ref{le:Pinveqass}) we have that there exists a $t_{0}>0$ such that $\|C(t)-I\|<c$ for all $t\in[0,t_{0})$, and by symmetry, for all $t\in(-t_{0},t_{0})$. Since $c<2$, we conclude that $\frac{1}{2}\|C(t)-I\|<\frac{c}{2}<1$, hence, $I+\frac{1}{2}(C(t)-I)=\frac{1}{2}(C(t)+I)$ is invertible with $\|(C(t)+I)^{-1}\|<\frac{1}{2-c}$ for all $t\in(-t_{0},t_{0})$. This implies that $-1\in\rho(C(t))$. By standard spectral theory it follows that the open ball centered at $-1$ with radius $\|R(-1,C(t))\|^{-1}$ is included in $\rho(C(t))$. Therefore,
\begin{equation}
\label{eq:prooflePinv}
	B_{\frac{2-c}{2}}(-1)\subset B_{\frac{1}{2\|R(-1,C(t))\|}}(-1)\subset \rho(C(t)) \quad \forall t\in(-t_{0},t_{0}),
\end{equation}
and by the analyticity of the resolvent, we have for $\mu\in B_{\frac{2-c}{2}}(-1)$ and $t\in(-t_{0},t_{0})$ that 
	\begin{align}\notag
		\|R(\mu,C(t))\| ={}& \left\|\sum_{n=0}^{\infty} (\mu+1)^{n}R(-1,C(t))^{n+1}\right\|\\
			{}&\leq 2\|R(-1,C(t))\| < \frac{2}{2-c}. \label{eq:prooflePinv1}
	\end{align}
 Since $\cosh(i\pi)=-1$, by continuity there exists ball in the
 complex plane with center $i\pi$ which is mapped under the $\cosh$
 inside the ball around $-1$. That is, there exists a $\tilde{r}>0$ such that 
 \begin{equation}\label{eq:prooflePinv2}
 \cosh(B_{\tilde{r}}(i\pi))\subset B_{\frac{1-c}{2}}(-1).
 \end{equation}
  Let $\lambda\in\mathbb{C}$ be such that
  $|\arg(\lambda)|\leq\frac{\pi}{2}$. We search for $s\in\mathbb{R}$
  such that $\lambda s\in B_{\tilde{r}}(i\pi)$. Let
  $s_{\lambda}=\frac{\pi \sin(\arg(\lambda))}{|\lambda|}$ be the
  unique element on the line $\left\{\lambda s:s\in\mathbb{R}\right\}$
  which is closest to $i\pi$. We have that  $|i\pi-\lambda
  s_{\lambda}|=\pi\cos(\arg(\lambda))$. Now, choose
  $\phi_{c}\in(0,\frac{\pi}{2})$ large enough such that
  $\pi\cos(\phi_{c})<\tilde{r}$ and choose $r_{c}>0$ such that $\frac{\pi}{r_c}<t_{0}$. Then, for all $\lambda^{2}\in\mathcal{R}_{c}$, we have that $\lambda s_{\lambda}\in B_{\tilde{r}}(i\pi)$ with $s_{\lambda}\in(-t_{0},t_{0})$. By (\ref{eq:prooflePinv2}), $\cosh(\lambda s_{\lambda})\in B_{\frac{2-c}{2}}(-1)$ and thus,
  \begin{equation}
  \cosh(\lambda s_{\lambda})\in\rho(C(s_{\lambda})), \quad \text{and} \quad \|R(\cosh(\lambda s_{\lambda}),C(s_{\lambda}))\|\leq \frac{2}{2-c},
  \end{equation}
  by (\ref{eq:prooflePinv}) and (\ref{eq:prooflePinv1}). Therefore, \ref{le:CosId4}.\ of Lemma \ref{le:CosId} implies that $\lambda^{2}\in\rho(A)$ and
  	\begin{align*}
	\| R(\lambda^{2},A)\| \leq{}& \sup_{t\in[0,|s_{\lambda}|]}\|C(t)\|\cdot  \frac{2|s_{\lambda}|e^{|s_{\lambda}\Re\lambda|}}{|\lambda|} \cdot \|R(\cosh(\lambda s),C(s_{\lambda}))\|\\
	\leq{}&\sup_{t\in[0,t_{0}]}\|C(t)\|\cdot  \frac{2\pi e^{\pi}}{|\lambda|^{2}} \cdot \frac{2}{2-c}\leq \frac{M_{c}}{|\lambda|^{2}}
	\end{align*}
 for some $M_{c}$ only depending on $\sup_{t\in[0,t_{0}]} \|C(t)\|$ and $c$.
\end{proof}

Combining the results from Lemmas \ref{le:CosLapl} and \ref{le:Pinv}
enable us to prove Theorem \ref{Tm:1.1}. As for semigroups we can prove a
slightly more general result.
\begin{Theorem}[Zero-two law for cosine families]
\label{thm:zerotwolaw}
Let $C$ be a strongly continuous cosine family on the Banach space
$X$. Denote by $A$ its infinitesimal generator. Then the following are equivalent
\begin{enumerate}
\item The following inequality holds
  \[
     \limsup_{t \rightarrow 0^+} \|C(t) - I\| < 2;
  \]
\item
  The following equality holds
  \[
     \limsup_{t \rightarrow 0^+} \|C(t) - I\| = 0;
  \]
 \item $A$ is a bounded operator.
\end{enumerate}
\end{Theorem}
\begin{proof}
Trivially the second item implies the first one. If the assertion in
part 3 holds, then the corresponding cosine family is given by $C(t) =
\sum_{n =0}^{\infty} A^n \frac{(-1)^n t^{2n}}{(2n)!}$. From this the
property in item 2 is easy to show. Hence it remains to show that item
1 implies item 3.

Let $c$ be the constant from equation (\ref{le:Pinveqass}), and let
$r_{c}>0,\phi_{c}\in[0,\frac{\pi}{2})$ be the constants from Lemma
\ref{le:Pinv}.  By Lemma \ref{le:CosLapl}, we have that there exists
$\omega'>\omega\geq0$ such that
\begin{equation}
\label{thm:Reseq1}
	\sup_{\lambda\in R_{\omega'} \cap S_{\phi_{c}}} \|\lambda^{2} R(\lambda^{2},A)\| <\infty,
\end{equation}
where
$R_{\omega'}=\left\{\lambda\in\mathbb{C}:\Re\lambda\geq\omega'\right\}$
and
$S_{\phi_{c}}=\left\{\mu\in\mathbb{C}:|\arg\mu|\leq\phi_{c}\right\}$. Now,
let $\lambda$ such that $|\lambda|>r_{c}$ and
$|\arg(\lambda)|\in(\phi_{c},\frac{\pi}{2}]$. Thus
$\lambda^{2}\in\mathcal{R}_{c}$, see (\ref{eq:Rc}), and so by Lemma \ref{le:Pinv}, 
\begin{equation}
\label{thm:Reseq2}
	\sup_{\lambda^{2}\in\mathcal{R}_{c}}\|\lambda^{2}R(\lambda^{2},A)\|<\infty.
\end{equation}
Let $f(z)=z^{2}$. It is easy to see that the closure of
$\mathbb{C}\setminus \left(\mathcal{R}_{c}\cup f(R_{\omega'}\cap
  S_{\phi_{c}})\right)$ is compact. Thus, (\ref{thm:Reseq1}) and
(\ref{thm:Reseq2}) yield that there exists an $R>0$ such that the
spectrum $\sigma(A)$ lies within the open ball $B_{R}(0)$ and
\begin{equation}
	\sup_{|\mu|>R}\|\mu R(\mu,A)\| <\infty.
\end{equation}
Hence we have that $\mu\mapsto R(\mu,A)$ has a removable singularity
at $\infty$. Since $A$ is closed, this implies that $A$ is a bounded
operator, \cite[Theorem I.6.13]{Kato}, and therefore part 3 is shown.
\end{proof}

\section{Similar laws on ${\mathbb R}$}
\label{sec:3}

In this previous sections we showed that uniform estimates in a
neighbourhood of zero implied additional properties. In this section we
study estimates which hold on ${\mathbb R}$ or $(0,\infty)$. We show
that by applying a scaling trick, the results can be obtained from the
already proved laws. The main theorem of this section is the following.
\begin{Theorem}
\label{tm:3.1}
The following assertions hold
\begin{enumerate}
	\item For a semigroup $T$ we have that (\ref{eq:1}) implies
          that $T(t)=I$ for all $t\geq 0$.
	\item If the strongly continuous cosine family $C$ on the
          Banach space $X$ satisfies
	\begin{equation}\label{assertCOS}
		\sup_{t\geq0}\|C(t)-I\|=r<2
	\end{equation}
  then $C(t)=I$ for all $t$. 
\end{enumerate}
\end{Theorem}
\begin{proof} Since the construction of the proof in the two items is
  very similar, we concentrate on the second one.

  For the Banach space $X$ we define $\ell^2({\mathbb N};X)$ as
  \begin{equation}
    \label{eq:5}
    \ell^2({\mathbb N};X) = \{ (x_n)_{n \in {\mathbb N}} \mid x_n \in
    X, \sum_{n \in {\mathbb N}} \|x_n\|^2 < \infty\}.
  \end{equation}
  With the norm 
  \[
     \|(x_n)\| = \sqrt{ \sum_{n \in {\mathbb N}} \|x_n\|^2}
  \]
  this is a Banach space. On this
  extended Banach space we define $C_{\mathrm{ext}}(t)$, $t \in
  {\mathbb R}$ as 
  \begin{equation}
    \label{eq:6}
      C_{\mathrm{ext}}(t) (x_n) = (C(nt) x_n).
  \end{equation}
  Hence it is a diagonal operator with scaled versions of $C$ on the
  diagonal. To prove that $C_{\mathrm{ext}}$ is strongly continuous,
  we take an arbitrary $x\in \ell^2({\mathbb N};X)$ and $t \in {\mathbb R}$. Furthermore,
  we choose an $\varepsilon >0$ and a $z = (z_n)$, with only finitely
  many $z_n$ unequal to zero, such that $\|x-z\| \leq \varepsilon$. By the
  construction of $\ell^2({\mathbb N};X)$ this is always possible. Now
  we find 
  \begin{align}
\nonumber
    \limsup_{h\rightarrow 0} \| C_{\mathrm{ext}} &(t+h) x
     -C_{\mathrm{ext}} (t) x \| \\
 \nonumber   \leq& \
   \limsup_{h\rightarrow 0} \left[ \| C_{\mathrm{ext}} (t+h) x
    -C_{\mathrm{ext}} (t+h) z \| + \right.\\
\nonumber 
   &\ \| C_{\mathrm{ext}} (t+h) z
    -C_{\mathrm{ext}} (t) z \| + 
    \left.
\| C_{\mathrm{ext}} (t) z
    -C_{\mathrm{ext}} (t) x \| \right]\\
  \leq&\
\label{eq:3.3}
  3 \|x-z\| + 3 \|z-x\| + \\
\nonumber 
  &\ \limsup_{h\rightarrow 0} \| C_{\mathrm{ext}} (t+h) z
    -C_{\mathrm{ext}} (t) z \|,
  \end{align}
  since by (\ref{assertCOS}), the cosine family $C_{\mathrm{ext}}$ is
  bounded by $3$. Let $N$ be such that $z_n=0$ for $n>N$. Then
  \begin{align*}
    \limsup_{h\rightarrow 0} \| C_{\mathrm{ext}} (t+h) z
    &-C_{\mathrm{ext}} (t) z \|^2 = \\
   &\ \limsup_{h\rightarrow 0}
    \sum_{n=1}^{N} \| C (nt+nh) z_n
    -C (nt) z_n \|^2 =0,
  \end{align*}
since $C$ is a strongly continuous cosine family. Combining this with
(\ref{eq:3.3}) we find that
\[
   \limsup_{h\rightarrow 0} \| C_{\mathrm{ext}} (t+h) x
     -C_{\mathrm{ext}} (t) x \| \leq 6 \varepsilon.
\]
Since $\varepsilon$ is arbitrarily, we conclude that
$C_{\mathrm{ext}}$ is a strongly continuous cosine family on
$\ell^2({\mathbb N};X)$.

Now we estimate the distance from this cosine family to the identity on
  $\ell^2({\mathbb N};X)$ for $t \in (0,1]$.
  \begin{align*}
    \|C_{\mathrm{ext}}(t) - I \|^2 = &\ \sup_{\|(x_n)\|=1}
    \|C_{\mathrm{ext}}(t) (x_n) - (x_n)\|^2 \\
   =&\    \sup_{\|(x_n)\|=1} \sum_{n \in {\mathbb N}} \|C(nt) x_n-
   x_n\|^2\\
   \leq& \sup_{\|(x_n)\|=1} \sum_{n \in {\mathbb N}} r^2 \|x_n\|^2 = r^2,
  \end{align*}
  where we have used (\ref{assertCOS}). In particular, this implies
  that 
  \[
    \limsup_{t \rightarrow 0^+} \|C_{\mathrm{ext}}(t) - I \| <2.
  \]
  By Theorem \ref{thm:zerotwolaw}, we conclude that the infinitesimal
  generator of $C_{\mathrm{ext}}$ is bounded. Since
  $C_{\mathrm{ext}}(t)$ is a diagonal operator, it is easy to see that
  its infinitesimal generator $A_{\mathrm{ext}}$ is diagonal as well. Furthermore, the
  $n$'th diagonal element equals $n A$. Since $n$ runs to infinity,
  $A_{\mathrm{ext}}$ can only be bounded when $A=0$. This immediately
  implies that $C(t)=I$ for all $t$.
\end{proof}

\mbox{}From the above proof it is clear that if Theorem
\ref{thm:zerotwolaw} would hold for non-strongly continuous cosine
families, then the strong continuity assumption can be removed from
item 2 in the above theorem as well.

As follows from the first item, for semigroups no continuity
assumption was needed. As mentioned in the introduction, this can also
be proved using operator algebraic result going back to Wallen
\cite{Wallen}.  In the following subsection, we present some
alternative proofs, showing that they can be asked as an exercise in a
first course on semigroup theory.

\subsection{Elementary proofs for semigroups}\label{sec:SG}
We now give some elementary proofs of the following result.
\begin{Theorem}\label{thm:SG1}
Let $T$ be a strongly continuous semigroup on the Banach space $X$, and
let $A$ denote its infinitesimal generator. If
\begin{equation}\label{eq:semigroupsup}
	r:=\sup_{t\geq0}\|T(t)-I\|<1,
\end{equation}
then $T(t)=I$ for all $t\geq0.$
\end{Theorem}
\begin{proof}[Proof of Theorem \ref{thm:SG1}, frequency domain]
Since the $C_{0}$-semigroup is bounded by (\ref{eq:semigroupsup}), $(0,\infty)\subset\rho(A)$ and for $\lambda >0$ we have that
\begin{eqnarray*}
   \| (\lambda I - A)^{-1}x_0 - \lambda ^{-1}x_0 \| &=& \left\|  \int_0^{\infty} \left(T(t) x_0 - x_0\right) e^{-\lambda t} dt\right\| \\
   &\leq& \int_0^{\infty} \left\|T(t) x_0 - x_0\right\| e^{-\lambda t} dt \leq \frac{r}{\lambda}\|x_0\|,
\end{eqnarray*}
where we used (\ref{eq:semigroupsup}).
Thus
\begin{equation*}
  \| \lambda (\lambda I - A)^{-1} - I \| \leq r
\end{equation*}
Since $r<1$, we know that $I + \left( \lambda (\lambda I - A)^{-1} - I\right)$ is boundedly invertible, and the norm of this inverse is less or equal to $(1-r)^{-1}$. Hence
\begin{equation}
  \label{eq:9}
  \|\lambda^{-1} (\lambda I- A)\| \leq \frac{1}{1-r}.
\end{equation}
So for all $\lambda >0$ we have that $\|I- \lambda^{-1} A\| \leq \frac{1}{1-r}$. This can only hold if $A=0$.
\end{proof}

\begin{proof}[Proof of Theorem \ref{thm:SG1}, time domain]
In general it holds that 
\begin{equation}
\label{eqnew}
	T(t)x-x=A\int_{0}^{t}T(s)x\ ds,\qquad t>0,x\in X.
\end{equation}
For $t>0$ let $B_{t}$ denote the bounded operator $x\mapsto B_{t}x:=\int_{0}^{t}T(s)x ds$. 
Since for all $x\in X$, 
\begin{equation*}
\|x-t^{-1}B_{t}x\|=\frac{1}{t}\left\|\int_{0}^{t}x-T(s)x\ ds\right\| \leq \frac{1}{t}\int_{0}^{t} \| x- T(s)x\|ds\leq r \|x\|,
\end{equation*}
and $r<1$, it follows that $t^{-1}B_{t}$ is boundedly invertible for all $t>0$ and 
\begin{equation}\label{eqnew2}
\|tB_{t}^{-1}\| \leq \frac{1}{1-r} \Leftrightarrow \|B_{t}^{-1}\|\leq \frac{1}{t(1-r)}.
\end{equation}
This concludes the proof because by (\ref{eqnew}) and the assumption, $\|AB_{t}\|\leq1$,
\begin{equation}
 \|A\|\leq \|B_{t}^{-1}\| \stackrel{(\ref{eqnew2})}{\leq} \frac{1}{t(1-r)} \qquad \forall t>0,
\end{equation}
hence, $A=0$.
\end{proof}

\subsection*{Acknowledgments}
We would like to thank A.~Bobwrowksi for introducing us to the problem whether $\sup_{t\geq0}\|C(t)-I\|<2$ implies $C(t)=I$. This was the starting point of this work. 
Furthermore, we are grateful to W.~Arendt for drawing the $0-3/2$ law
in \cite[Theorem 1.1 in Three Line Proofs]{UlmerSeminare2012} to our
attention.

\end{document}